\documentclass[11pt,fleqn]{article}

\usepackage{amsmath,amssymb,amsthm,esint}
\usepackage{geometry}
\usepackage[pdfpagemode=UseNone,pdfstartview=FitH]{hyperref}

\newcommand{\bdry}[1]{\partial #1}

\newcommand{\dint}{\ds{\int}}
\newcommand{\dist}[2]{\text{dist}\, (#1,#2)}
\newcommand{\ds}[1]{\displaystyle #1}
\newcommand{\eps}{\varepsilon}

\newcommand{\id}[1]{id_{#1}}
\newcommand{\incl}{\subset}

\newcommand{\norm}[2][]{\left\|#2\right\|_{#1}}
\renewcommand{\O}{\text{O}}
\renewcommand{\o}{\text{o}}
\newcommand{\PS}[1]{$(\text{PS})_{#1}$}
\newcommand{\QED}{\mbox{\qedhere}}
\newcommand{\restr}[2]{\left.#1\right|_{#2}}
\newcommand{\seq}[1]{\left(#1\right)}
\newcommand{\set}[1]{\left\{#1\right\}}

\newcommand{\N}{\mathbb N}
\newcommand{\R}{\mathbb R}
\newcommand{\RP}{\R \text{P}}
\newcommand{\Z}{\mathbb Z}

\newcommand{\A}{{\cal A}}
\newcommand{\F}{{\cal F}}

\def\cprime{$'$}

\DeclareMathOperator{\divg}{div}

\newenvironment{enumroman}{\begin{enumerate}
\renewcommand{\theenumi}{$(\roman{enumi})$}
\renewcommand{\labelenumi}{$(\roman{enumi})$}}{\end{enumerate}}

\newenvironment{properties}[1]{\begin{enumerate}
\renewcommand{\theenumi}{$(#1_\arabic{enumi})$}
\renewcommand{\labelenumi}{$(#1_\arabic{enumi})$}}{\end{enumerate}}

\newtheorem{lemma}{Lemma}[section]
\newtheorem{proposition}[lemma]{Proposition}
\newtheorem{theorem}[lemma]{Theorem}

\theoremstyle{remark}
\newtheorem{remark}[lemma]{Remark}

\numberwithin{equation}{section}

\title{\bf On a class of critical $(p,q)$-Laplacian problems}
\author{\bf Pasquale Candito\\
Universit\`{a} degli Studi di Reggio Calabria\\
89100 Reggio Calabria, Italy\\
\em pasquale.candito@unirc.it\\
[\bigskipamount]
\bf Salvatore A. Marano\\
Universit\`{a} degli Studi di Catania\\
95125 Catania, Italy\\
\em marano@dmi.unict.it\\
[\bigskipamount]
\bf Kanishka Perera\thanks{This work was completed while the third-named author was visiting Universit\`{a} di Reggio Calabria, and he is grateful for the kind hospitality of the host university.}\\
Florida Institute of Technology\\
Melbourne, FL 32901, USA\\
\em kperera@fit.edu}
\date{}

\begin{document}

\maketitle

\begin{abstract}
We obtain nontrivial solutions of a critical $(p,q)$-Laplacian problem in a bounded domain. In addition to the usual difficulty of the loss of compactness associated with problems involving critical Sobolev exponents, this problem lacks a direct sum decomposition suitable for applying the classical linking theorem. We show that every Palais-Smale sequence at a level below a certain energy threshold admits a subsequence that converges weakly to a nontrivial critical point of the variational functional. Then we prove an abstract critical point theorem based on a cohomological index and use it to construct a minimax level below this threshold.
\end{abstract}

\section{Introduction and main results}

The $(p,q)$-Laplacian operator
\[
\Delta_p\, u + \Delta_q\, u = \divg \left[\left(|\nabla u|^{p-2} + |\nabla u|^{q-2}\right) \nabla u\right]
\]
appears in a wide range of applications that include biophysics \cite{MR527914}, plasma physics \cite{MR901723}, reaction-diffusion equations \cite{MR518461,MR2126276}, and models of elementary particles \cite{MR0174304,MR1626832,MR1785469}. Consequently, quasilinear elliptic boundary value problems involving this operator have been widely studied in the literature (see, e.g., \cite{MR1929880,MR2536276,MR2834776,MR2988766} and the references therein). In particular, the critical $(p,q)$-Laplacian problem
\[
\left\{\begin{aligned}
- \Delta_p\, u - \Delta_q\, u & = \mu\, |u|^{r-2}\, u + |u|^{p^\ast - 2}\, u && \text{in } \Omega\\[10pt]
u & = 0 && \text{on } \bdry{\Omega},
\end{aligned}\right.
\]
where $\Omega$ is a bounded domain in $\R^N,\, N > p > q > 1$, $\mu > 0$, and $p^\ast = Np/(N - p)$ is the critical Sobolev exponent, has been studied by Li and Zhang \cite{MR2509998} in the case $1 < r < q$ and by Yin and Yang \cite{MR2890966} in the case $p < r < p^\ast$. In the present paper we consider the question of existence of nontrivial solutions in the borderline case
\begin{equation} \label{1}
\left\{\begin{aligned}
- \Delta_p\, u - \Delta_q\, u & = \mu\, |u|^{q-2}\, u + \lambda\, |u|^{p-2}\, u + |u|^{p^\ast - 2}\, u && \text{in } \Omega\\[10pt]
u & = 0 && \text{on } \bdry{\Omega}
\end{aligned}\right.
\end{equation}
with $\mu \in \R$ and $\lambda > 0$. In addition to the usual difficulty of the lack of compactness associated with problems involving critical exponents, this problem is further complicated by the absence of a direct sum decomposition suitable for applying the linking theorem when $\mu$ is above the second eigenvalue of the eigenvalue problem
\begin{equation} \label{7}
\left\{\begin{aligned}
- \Delta_q\, u & = \mu\, |u|^{q-2}\, u && \text{in } \Omega\\[10pt]
u & = 0 && \text{on } \bdry{\Omega}.
\end{aligned}\right.
\end{equation}
To overcome this difficulty, we will first prove an abstract critical point theorem based on a cohomological index that generalizes the classical linking theorem of Rabinowitz \cite{MR0488128}.

Weak solutions of problem \eqref{1} coincide with critical points of the $C^1$-functional
\begin{equation} \label{2}
\Phi(u) = \int_\Omega \left(\frac{1}{p}\, |\nabla u|^p + \frac{1}{q}\, |\nabla u|^q - \frac{\mu}{q}\, |u|^q - \frac{\lambda}{p}\, |u|^p - \frac{1}{p^\ast}\, |u|^{p^\ast}\right) dx, \quad u \in W^{1,p}_0(\Omega),
\end{equation}
where $W^{1,p}_0(\Omega)$ is the usual Sobolev space with the norm $\norm{u} = \norm[p]{\nabla u}$ and $\norm[p]{\cdot}$ denotes the norm in $L^p(\Omega)$. Recall that $\Phi$ satisfies the Palais-Smale compactness condition at the level $c \in \R$, or \PS{c} for short, if every sequence $\seq{u_j} \subset W^{1,p}_0(\Omega)$ such that $\Phi(u_j) \to c$ and $\Phi'(u_j) \to 0$, called a \PS{c} sequence, has a convergent subsequence. Let
\begin{equation} \label{3}
S = \inf_{u \in W^{1,p}_0(\Omega) \setminus \set{0}}\, \frac{\norm[p]{\nabla u}^p}{\norm[p^\ast]{u}^p} > 0
\end{equation}
be the best constant for the Sobolev imbedding $W^{1,p}_0(\Omega) \hookrightarrow L^{p^\ast}(\Omega)$. Our existence results will be based on the following proposition.

\begin{proposition} \label{Proposition 1}
If $c < S^{N/p}/N$ and $c \ne 0$, then every {\em \PS{c}} sequence has a subsequence that converges weakly to a nontrivial critical point of $\Phi$.
\end{proposition}

Let
\begin{equation} \label{6}
\mu_1 = \inf_{u \in W^{1,q}_0(\Omega) \setminus \set{0}}\, \frac{\norm[q]{\nabla u}^q}{\norm[q]{u}^q} > 0
\end{equation}
be the first eigenvalue of the eigenvalue problem \eqref{7}. First we seek a nonnegative nontrivial solution of problem \eqref{1} when $\mu \le \mu_1$. Let
\begin{equation} \label{8}
\lambda_1 = \inf_{u \in W^{1,p}_0(\Omega) \setminus \set{0}}\, \frac{\norm[p]{\nabla u}^p}{\norm[p]{u}^p} > 0
\end{equation}
be the first eigenvalue of the eigenvalue problem
\[
\left\{\begin{aligned}
- \Delta_p\, u & = \lambda\, |u|^{p-2}\, u && \text{in } \Omega\\[10pt]
u & = 0 && \text{on } \bdry{\Omega}.
\end{aligned}\right.
\]
Our first main result is the following theorem.

\begin{theorem} \label{Theorem 2}
Assume that $1 < q < p$ and $p^2 < N$. If $0 < \lambda < \lambda_1$ and $\mu \le \mu_1$, then problem \eqref{1} has a nonnegative nontrivial solution in each of the following cases:
\begin{enumroman}
\item \label{Theorem 4.i} $N(p - 1)/(N - p) \le q < (N - p)\, p/N$,
\item \label{Theorem 4.ii} $N(p - 1)/(N - 1) < q < \min \set{N(p - 1)/(N - p),(N - p)\, p/N}$,
\item \label{Theorem 4.iii} $(1 - 1/N)\, p^2 + p < N$ and $q = N(p - 1)/(N - 1)$,
\item \label{Theorem 4.iv} $(p - 1)\, p^2/(N - p) < q < N(p - 1)/(N - 1)$.
\end{enumroman}
\end{theorem}

Now we assume that $p < q^\ast$, where $q^\ast = Nq/(N - q)$ is the critical exponent for the imbedding $W^{1,q}_0(\Omega) \hookrightarrow L^p(\Omega)$. Then we have the following theorem.

\begin{theorem} \label{Theorem 5}
Assume that $1 < q < p < \min \set{N,q^\ast}$. If $\mu < \mu_1$, then there exists $\lambda^\ast(\mu) > 0$ such that problem \eqref{1} has a nonnegative nontrivial solution for all $\lambda \ge \lambda^\ast(\mu)$.
\end{theorem}

Let $u^\pm(x) = \max \set{\pm u(x),0}$ be the positive and negative parts of $u$, respectively, and set
\[
\Phi^+(u) = \int_\Omega \left(\frac{1}{p}\, |\nabla u|^p + \frac{1}{q}\, |\nabla u|^q - \frac{\mu}{q}\, (u^+)^q - \frac{\lambda}{p}\, (u^+)^p - \frac{1}{p^\ast}\, (u^+)^{p^\ast}\right) dx, \quad u \in W^{1,p}_0(\Omega).
\]
If $u$ is a critical point of $\Phi^+$, then
\[
{\Phi^+}'(u)\, u^- = \int_\Omega \left(|\nabla u^-|^p + |\nabla u^-|^q\right) dx = 0
\]
and hence $u^- = 0$, so $u = u^+$ is a critical point of $\Phi$ and therefore a nonnegative solution of problem \eqref{1}. If, in addition, $\mu \ge 0$, then $u > 0$ in $\Omega$ by Cianchi \cite[Theorem 2]{MR1469584}, Lieberman \cite[Theorem 1.7]{MR1104103}, and Pucci and Serrin \cite[Theorem 1.1.1]{MR2356201}. Proofs of Theorems \ref{Theorem 2} and \ref{Theorem 5} will be based on constructing minimax levels of mountain pass type for $\Phi^+$ below the threshold level given in Proposition \ref{Proposition 1}.

Next we seek a (possibly nodal) nontrivial solution of problem \eqref{1} when $\mu \ge \mu_1$. We have the following theorem.

\begin{theorem} \label{Theorem 6}
Assume that $1 < q < p < \min \set{N,q^\ast}$. If $\mu \ge \mu_1$, then there exists $\lambda^\ast(\mu) > 0$ such that problem \eqref{1} has a nontrivial solution for all $\lambda \ge \lambda^\ast(\mu)$.
\end{theorem}

This extension of Theorem \ref{Theorem 5} is nontrivial. Indeed, the functional $\Phi$ does not have the mountain pass geometry when $\mu \ge \mu_1$ since the origin is no longer a local minimizer, and a linking type argument is needed. However, the classical linking theorem cannot be used since the nonlinear operator $- \Delta_q$ does not have linear eigenspaces. We will use a more general construction based on sublevel sets as in Perera and Szulkin \cite{MR2153141} (see also Perera et al.\! \cite[Proposition 3.23]{MR2640827}). Moreover, the standard sequence of eigenvalues of $- \Delta_q$ based on the genus does not give enough information about the structure of the sublevel sets to carry out this linking construction. Therefore we will use a different sequence of eigenvalues introduced in Perera \cite{MR1998432} that is based on a cohomological index.

The $\Z_2$-cohomological index of Fadell and Rabinowitz \cite{MR57:17677} is defined as follows. Let $W$ be a Banach space and let $\A$ denote the class of symmetric subsets of $W \setminus \set{0}$. For $A \in \A$, let $\overline{A} = A/\Z_2$ be the quotient space of $A$ with each $u$ and $-u$ identified, let $f : \overline{A} \to \RP^\infty$ be the classifying map of $\overline{A}$, and let $f^\ast : H^\ast(\RP^\infty) \to H^\ast(\overline{A})$ be the induced homomorphism of the Alexander-Spanier cohomology rings. The cohomological index of $A$ is defined by
\[
i(A) = \begin{cases}
\sup \set{m \ge 1 : f^\ast(\omega^{m-1}) \ne 0}, & A \ne \emptyset\\[5pt]
0, & A = \emptyset,
\end{cases}
\]
where $\omega \in H^1(\RP^\infty)$ is the generator of the polynomial ring $H^\ast(\RP^\infty) = \Z_2[\omega]$. For example, the classifying map of the unit sphere $S^{m-1}$ in $\R^m,\, m \ge 1$ is the inclusion $\RP^{m-1} \incl \RP^\infty$, which induces isomorphisms on $H^q$ for $q \le m - 1$, so $i(S^{m-1}) = m$. The following proposition summarizes the basic properties of this index.

\begin{proposition}[Fadell-Rabinowitz \cite{MR57:17677}] \label{Proposition 7}
The index $i : \A \to \N \cup \set{0,\infty}$ has the following properties:
\begin{properties}{i}
\item Definiteness: $i(A) = 0$ if and only if $A = \emptyset$;
\item \label{i2} Monotonicity: If there is an odd continuous map from $A$ to $B$ (in particular, if $A \subset B$), then $i(A) \le i(B)$. Thus, equality holds when the map is an odd homeomorphism;
\item Dimension: $i(A) \le \dim W$;
\item Continuity: If $A$ is closed, then there is a closed neighborhood $N \in \A$ of $A$ such that $i(N) = i(A)$. When $A$ is compact, $N$ may be chosen to be a $\delta$-neighborhood $N_\delta(A) = \set{u \in W : \dist{u}{A} \le \delta}$;
\item Subadditivity: If $A$ and $B$ are closed, then $i(A \cup B) \le i(A) + i(B)$;
\item \label{i6} Stability: If $SA$ is the suspension of $A \ne \emptyset$, obtained as the quotient space of $A \times [-1,1]$ with $A \times \set{1}$ and $A \times \set{-1}$ collapsed to different points, then $i(SA) = i(A) + 1$;
\item \label{i7} Piercing property: If $A$, $A_0$ and $A_1$ are closed, and $\varphi : A \times [0,1] \to A_0 \cup A_1$ is a continuous map such that $\varphi(-u,t) = - \varphi(u,t)$ for all $(u,t) \in A \times [0,1]$, $\varphi(A \times [0,1])$ is closed, $\varphi(A \times \set{0}) \subset A_0$ and $\varphi(A \times \set{1}) \subset A_1$, then $i(\varphi(A \times [0,1]) \cap A_0 \cap A_1) \ge i(A)$;
\item Neighborhood of zero: If $U$ is a bounded closed symmetric neighborhood of $0$, then $i(\bdry{U}) = \dim W$.
\end{properties}
\end{proposition}

The Dirichlet spectrum of $- \Delta_q$ in $\Omega$ consists of those $\mu \in \R$ for which problem \eqref{7} has a nontrivial solution. Although a complete description of the spectrum is not yet known when $N \ge 2$, we can define an increasing and unbounded sequence of eigenvalues via a suitable minimax scheme. The standard scheme based on the genus does not give the index information necessary to prove Theorem \ref{Theorem 6}, so we will use the following scheme based on the cohomological index as in Perera \cite{MR1998432}. Let
\[
\Psi(u) = \frac{1}{\dint_\Omega |u|^q\, dx}, \quad u \in S_q = \set{u \in W^{1,q}_0(\Omega) : \int_\Omega |\nabla u|^q\, dx = 1}.
\]
Then eigenvalues of problem \eqref{7} on $S_q$ coincide with critical values of $\Psi$. We use the standard notation
\[
\Psi^a = \set{u \in S_q : \Psi(u) \le a}, \quad \Psi_a = \set{u \in S_q : \Psi(u) \ge a}, \quad a \in \R
\]
for the sublevel sets and superlevel sets, respectively. Let $\F$ denote the class of symmetric subsets of $S_q$ and set
\[
\mu_k := \inf_{M \in \F,\; i(M) \ge k}\, \sup_{u \in M}\, \Psi(u), \quad k \in \N.
\]
Then $0 < \mu_1 < \mu_2 \le \mu_3 \le \cdots \to + \infty$ is a sequence of eigenvalues of problem \eqref{7} and
\begin{equation} \label{15}
\mu_k < \mu_{k+1} \implies i(\Psi^{\mu_k}) = i(S_q \setminus \Psi_{\mu_{k+1}}) = k
\end{equation}
(see Perera et al.\! \cite[Propositions 3.52 and 3.53]{MR2640827}).

Proof of Theorem \ref{Theorem 6} will make essential use of \eqref{15} and will be based on the following abstract critical point theorem, which is of independent interest. Let $W$ be a Banach space, let
\[
S = \set{u \in W : \norm{u} = 1}
\]
be the unit sphere in $W$, and let
\[
\pi : W \setminus \set{0} \to S, \quad u \mapsto \frac{u}{\norm{u}}
\]
be the radial projection onto $S$.

\begin{theorem} \label{Theorem 8}
Let $\Phi$ be a $C^1$-functional on $W$ and let $A_0,\, B_0$ be disjoint nonempty closed symmetric subsets of $S$ such that
\begin{equation} \label{16}
i(A_0) = i(S \setminus B_0) < \infty.
\end{equation}
Assume that there exist $R > r > 0$ and $v \in S \setminus A_0$ such that
\[
\sup \Phi(A) \le \inf \Phi(B), \qquad \sup \Phi(X) < \infty,
\]
where
\begin{gather*}
A = \set{tu : u \in A_0,\, 0 \le t \le R} \cup \set{R\, \pi((1 - t)\, u + tv) : u \in A_0,\, 0 \le t \le 1},\\[5pt]
B = \set{ru : u \in B_0},\\[5pt]
X = \set{tu : u \in A,\, \norm{u} = R,\, 0 \le t \le 1}.
\end{gather*}
Let $\Gamma = \set{\gamma \in C(X,W) : \gamma(X) \text{ is closed and} \restr{\gamma}{A} = \id{A}}$ and set
\[
c := \inf_{\gamma \in \Gamma}\, \sup_{u \in \gamma(X)}\, \Phi(u).
\]
Then
\[
\inf \Phi(B) \le c \le \sup \Phi(X)
\]
and $\Phi$ has a {\em \PS{c}} sequence.
\end{theorem}

\begin{remark}
Theorem \ref{Theorem 8}, which does not require a direct sum decomposition, generalizes the linking theorem of Rabinowitz \cite{MR0488128}.
\end{remark}

\section{Preliminaries}

In this preliminary section we prove Proposition \ref{Proposition 1} and Theorem \ref{Theorem 8}.

\begin{proof}[Proof of Proposition \ref{Proposition 1}]
Let $\seq{u_j}$ be a \PS{c} sequence. Then
\begin{equation} \label{4}
\Phi(u_j) = \int_\Omega \left(\frac{1}{p}\, |\nabla u_j|^p + \frac{1}{q}\, |\nabla u_j|^q - \frac{\mu}{q}\, |u_j|^q - \frac{\lambda}{p}\, |u_j|^p - \frac{1}{p^\ast}\, |u_j|^{p^\ast}\right) dx = c + \o(1)
\end{equation}
and
\begin{equation} \label{5}
\Phi'(u_j)\, u_j = \int_\Omega \left(|\nabla u_j|^p + |\nabla u_j|^q - \mu\, |u_j|^q - \lambda\, |u_j|^p - |u_j|^{p^\ast}\right) dx = \o(1) \norm{u_j}.
\end{equation}
So
\[
\int_\Omega \left[\left(\frac{1}{q} - \frac{1}{p}\right)\! \left(|\nabla u_j|^q - \mu\, |u_j|^q\right) + \left(\frac{1}{p} - \frac{1}{p^\ast}\right) |u_j|^{p^\ast}\right] dx = \o(1) \norm{u_j} + \O(1),
\]
and since $q < p < p^\ast$, this and the H\"{o}lder and Young inequalities yield
\[
\int_\Omega |u_j|^{p^\ast}\, dx \le \o(1) \norm{u_j} + \O(1).
\]
Since $p > 1$, it follows from this and \eqref{4} that $\seq{u_j}$ is bounded in $W^{1,p}_0(\Omega)$. So a renamed subsequence converges to some $u$ weakly in $W^{1,p}_0(\Omega)$, strongly in $L^s(\Omega)$ for all $1 \le s < p^\ast$, and a.e.\! in $\Omega$. Then $u$ is a critical point of $\Phi$ by the weak continuity of $\Phi'$.

Suppose $u = 0$. Since $\seq{u_j}$ is bounded in $W^{1,p}_0(\Omega)$ and converges to $0$ in $L^p(\Omega)$, \eqref{5} gives
\[
\o(1) = \int_\Omega \left(|\nabla u_j|^p + |\nabla u_j|^q - |u_j|^{p^\ast}\right) dx \ge \norm{u_j}^p \left(1 - \frac{\norm{u_j}^{p^\ast - p}}{S^{p^\ast/p}}\right)
\]
by \eqref{3}. If $\norm{u_j} \to 0$, then $\Phi(u_j) \to 0$, contradicting $c \ne 0$, so this implies
\[
\norm{u_j}^p \ge S^{N/p} + \o(1)
\]
for a renamed subsequence. Then \eqref{4} and \eqref{5} yield
\[
c = \int_\Omega \left[\left(\frac{1}{p} - \frac{1}{p^\ast}\right) |\nabla u_j|^p + \left(\frac{1}{q} - \frac{1}{p^\ast}\right) |\nabla u_j|^q\right] dx + \o(1) \ge \frac{S^{N/p}}{N} + \o(1),
\]
contradicting $c < S^{N/p}/N$.
\end{proof}

\begin{proof}[Proof of Theorem \ref{Theorem 8}]
First we show that $A$ (homotopically) links $B$ with respect to $X$ in the sense that
\begin{equation} \label{17}
\gamma(X) \cap B \ne \emptyset \quad \forall \gamma \in \Gamma.
\end{equation}
If \eqref{17} does not hold, then there is a map $\gamma \in C(X,W \setminus B)$ such that $\gamma(X)$ is closed and $\restr{\gamma}{A} = \id{A}$. Let
\[
\widetilde{A} = \set{R\, \pi((1 - |t|)\, u + tv) : u \in A_0,\, -1 \le t \le 1}
\]
and note that $\widetilde{A}$ is closed since $A_0$ is closed (here $(1 - |t|)\, u + tv \ne 0$ since $v$ is not in the symmetric set $A_0$). Since
\[
SA_0 \to \widetilde{A}, \quad (u,t) \mapsto R\, \pi((1 - |t|)\, u + tv)
\]
is an odd continuous map,
\begin{equation} \label{18}
i(\widetilde{A}) \ge i(SA_0) = i(A_0) + 1
\end{equation}
by \ref{i2} and \ref{i6} of Proposition \ref{Proposition 7}. Consider the map
\[
\varphi : \widetilde{A} \times [0,1] \to W \setminus B, \quad \varphi(u,t) = \begin{cases}
\gamma(tu), & u \in \widetilde{A} \cap A\\[5pt]
- \gamma(-tu), & u \in \widetilde{A} \setminus A,
\end{cases}
\]
which is continuous since $\gamma$ is the identity on the symmetric set $\set{tu : u \in A_0,\, 0 \le t \le R}$. We have $\varphi(-u,t) = - \varphi(u,t)$ for all $(u,t) \in \widetilde{A} \times [0,1]$, $\varphi(\widetilde{A} \times [0,1]) = \gamma(X) \cup (- \gamma(X))$ is closed, and $\varphi(\widetilde{A} \times \set{0}) = \set{0}$ and $\varphi(\widetilde{A} \times \set{1}) = \widetilde{A}$ since $\restr{\gamma}{A} = \id{A}$. Applying \ref{i7} with $\widetilde{A}_0 = \set{u \in W : \norm{u} \le r}$ and $\widetilde{A}_1 = \set{u \in W : \norm{u} \ge r}$ gives
\begin{equation} \label{19}
i(\widetilde{A}) \le i(\varphi(\widetilde{A} \times [0,1]) \cap \widetilde{A}_0 \cap \widetilde{A}_1) \le i((W \setminus B) \cap S_r) = i(S_r \setminus B) = i(S \setminus B_0),
\end{equation}
where $S_r = \set{u \in W : \norm{u} = r}$. By \eqref{18} and \eqref{19}, $i(A_0) < i(S \setminus B_0)$, contradicting \eqref{16}. Hence \eqref{17} holds.

It follows from \eqref{17} that $c \ge \inf \Phi(B)$, and $c \le \sup \Phi(X)$ since $\id{X} \in \Gamma$. By a standard argument, $\Phi$ has a \PS{c} sequence (see, e.g., Ghoussoub \cite{MR1251958}).
\end{proof}

\begin{remark}
The linking construction in the above proof was used in Perera and Szulkin \cite{MR2153141} to obtain nontrivial solutions of $p$-Laplacian problems with nonlinearities that interact with the spectrum. A similar construction based on the notion of cohomological linking was given in Degiovanni and Lancelotti \cite{MR2371112}. See also Perera et al.\! \cite[Proposition 3.23]{MR2640827}.
\end{remark}

\section{Proofs of Theorems \ref{Theorem 2} and \ref{Theorem 5}}

Fix $u_0 > 0$ in $W^{1,p}_0(\Omega)$ such that $\norm[p^\ast]{u_0} = 1$. Since $q < p < p^\ast$,
\[
\Phi^+(t u_0) = \int_\Omega \left(\frac{t^p}{p}\, |\nabla u_0|^p + \frac{t^q}{q}\, |\nabla u_0|^q - \frac{\mu\, t^q}{q}\, u_0^q - \frac{\lambda\, t^p}{p}\, u_0^p\right) dx - \frac{t^{p^\ast}}{p^\ast} \to - \infty
\]
as $t \to + \infty$. Take $t_0 > 0$ so large that $\Phi^+(t_0 u_0) \le 0$, let
\[
\Gamma = \set{\gamma \in C([0,1],W^{1,p}_0(\Omega)) : \gamma(0) = 0,\, \gamma(1) = t_0 u_0}
\]
be the class of paths joining $0$ and $t_0 u_0$, and set
\[
c := \inf_{\gamma \in \Gamma}\, \max_{u \in \gamma([0,1])}\, \Phi^+(u).
\]

\begin{lemma} \label{Lemma 2}
If $0 < c < S^{N/p}/N$, then problem \eqref{1} has a nonnegative nontrivial solution.
\end{lemma}

\begin{proof}
By the mountain pass theorem, $\Phi^+$ has a \PS{c} sequence $\seq{u_j}$. An argument similar to that in the proof of Proposition \ref{Proposition 1} shows that a subsequence of $\seq{u_j}$ converges weakly to a nontrivial critical point $u$ of $\Phi^+$.
\end{proof}

We have the following upper bounds for $c$.

\begin{lemma} \label{Lemma 3}
Let $\widetilde{\lambda} = \lambda/2$.
\begin{enumroman}
\item \label{Lemma 3.i} If $\int_\Omega |\nabla u_0|^p\, dx > \widetilde{\lambda} \int_\Omega u_0^p\, dx$, then
\[
c \le \frac{1}{N} \left[\int_\Omega \big(|\nabla u_0|^p - \widetilde{\lambda}\, u_0^p\big)\, dx\right]^{N/p} + \left(\frac{1}{q} - \frac{1}{p}\right) \frac{\left[\dint_\Omega \big(|\nabla u_0|^q - \mu\, u_0^q\big)\, dx\right]^{p/(p-q)}}{\left(\widetilde{\lambda} \dint_\Omega u_0^p\, dx\right)^{q/(p-q)}}.
\]
\item \label{Lemma 3.ii} If $\int_\Omega |\nabla u_0|^p\, dx \le \widetilde{\lambda} \int_\Omega u_0^p\, dx$, then
\[
c \le \left(\frac{1}{q} - \frac{1}{p}\right) \frac{\left[\dint_\Omega \big(|\nabla u_0|^q - \mu\, u_0^q\big)\, dx\right]^{p/(p-q)}}{\left(\widetilde{\lambda} \dint_\Omega u_0^p\, dx\right)^{q/(p-q)}}.
\]
\end{enumroman}
\end{lemma}

\begin{proof}
Since $\gamma(s) = s t_0 u_0$ is a path in $\Gamma$,
\begin{multline*}
c \le \max_{s \in [0,1]}\, \Phi^+(s t_0 u_0) \le \max_{t \ge 0}\, \Phi^+(t u_0) \le \max_{t \ge 0}\, \left[\frac{t^p}{p} \int_\Omega \big(|\nabla u_0|^p - \widetilde{\lambda}\, u_0^p\big)\, dx - \frac{t^{p^\ast}}{p^\ast}\right]\\[5pt]
+ \max_{t \ge 0}\, \left[\frac{t^q}{q} \int_\Omega \big(|\nabla u_0|^q - \mu\, u_0^q\big)\, dx - \frac{\widetilde{\lambda}\, t^p}{p} \int_\Omega u_0^p\, dx\right]. \QED
\end{multline*}
\end{proof}

\begin{proof}[Proof of Theorem \ref{Theorem 2}]
Without loss of generality we may assume that $0 \in \Omega$. Let $r > 0$ be so small that $B_{2r}(0) \subset \Omega$, take a function $\psi \in C^\infty_0(B_{2r}(0),[0,1])$ such that $\psi = 1$ on $B_r(0)$, and set
\[
u_\eps(x) = \frac{\psi(x)}{\left(\eps + |x|^{p/(p-1)}\right)^{(N-p)/p}}, \qquad v_\eps(x) = \frac{u_\eps(x)}{\norm[p^\ast]{u_\eps}}
\]
for $\eps > 0$. Then $\norm[p^\ast]{v_\eps} = 1$ and
\begin{gather}
\label{9} \int_\Omega |\nabla v_\eps|^p\, dx = S + \O(\eps^{(N-p)/p}),\\[5pt]
\int_\Omega v_\eps^p\, dx = \begin{cases}
K \eps^{p-1} + \O(\eps^{(N-p)/p}), & p^2 < N\\[5pt]
K \eps^{p-1}\, |\log \eps| + \O(\eps^{p-1}), & p^2 = N\\[5pt]
\O(\eps^{(N-p)/p}), & p^2 > N
\end{cases}
\intertext{for some constant $K > 0$,}
\int_\Omega |\nabla v_\eps|^q\, dx = \begin{cases}
\O(\eps^{N(p-1)(p-q)/p^2}), & q > \frac{N(p - 1)}{N - 1}\\[7.5pt]
\O(\eps^{N(N-p)(p-1)/(N-1)\, p^2}\, |\log \eps|), & q = \frac{N(p - 1)}{N - 1}\\[7.5pt]
\O(\eps^{(N-p)\, q/p^2}), & q < \frac{N(p - 1)}{N - 1},
\end{cases}
\intertext{and}
\label{12} \int_\Omega v_\eps^q\, dx = \begin{cases}
\O(\eps^{(p-1)[Np-(N-p)\, q]/p^2}), & q > \frac{N(p - 1)}{N - p}\\[7.5pt]
\O(\eps^{N(p-1)/p^2}\, |\log \eps|), & q = \frac{N(p - 1)}{N - p}\\[7.5pt]
\O(\eps^{(N-p)\, q/p^2}), & q < \frac{N(p - 1)}{N - p}
\end{cases}
\end{gather}
as $\eps \to 0$ (see, e.g., Dr{\'a}bek and Huang \cite{MR1473856}). We apply Lemma \ref{Lemma 2} with $u_0 = v_\eps$. Since $\mu \le \mu_1$ and by \eqref{6}, \eqref{8}, and \eqref{3},
\[
\Phi^+(u) \ge \frac{1}{p} \left(1 - \frac{\lambda}{\lambda_1}\right) \norm{u}^p - \frac{1}{p^\ast}\, S^{- p^\ast/p} \norm{u}^{p^\ast} \quad \forall u \in W^{1,p}_0(\Omega).
\]
Since $\lambda < \lambda_1$ and $p^\ast > p$, it follows from this that $0$ is a strict local minimizer of $\Phi^+$, so $c > 0$. We will verify that in each case $c < S^{N/p}/N$ for $\eps > 0$ sufficiently small by using Lemma \ref{Lemma 3} \ref{Lemma 3.i} and \eqref{9}--\eqref{12}.

\ref{Theorem 4.i} Since $p^2 < N$ and $q \ge N(p - 1)/(N - p) > N(p - 1)/(N - 1)$, we have
\begin{equation} \label{13}
c \le \frac{1}{N} \left[S - K \widetilde{\lambda} \eps^{p-1} + \O(\eps^{(N-p)/p})\right]^{N/p} + \O(\eps^{(p-1)[N/p-q/(p-q)]}).
\end{equation}
$(N - p)/p > p - 1$ since $p^2 < N$, and $(p - 1)[N/p - q/(p - q)] > p - 1$ since $q < (N - p)\, p/N$, so the desired conclusion follows.

\ref{Theorem 4.ii} Since $N(p - 1)/(N - 1) < q < N(p - 1)/(N - p)$, \eqref{13} still holds, and $(p - 1)[N/p - q/(p - q)] > p - 1$ since $q < (N - p)\, p/N$.

\ref{Theorem 4.iii} Since $q = N(p - 1)/(N - 1) < N(p - 1)/(N - p)$, we have
\[
c \le \frac{1}{N} \left[S - K \widetilde{\lambda} \eps^{p-1} + \O(\eps^{(N-p)/p})\right]^{N/p} + \O(\eps^{N(N-p^2)(p-1)/(N-p)\, p}\, |\log \eps|^{(N-1)\, p/(N-p)}),
\]
and $N(N - p^2)(p - 1)/(N - p)\, p > p - 1$ since $(1 - 1/N)\, p^2 + p < N$.

\ref{Theorem 4.iv} Since $q < N(p - 1)/(N - 1) < N(p - 1)/(N - p)$, we have
\[
c \le \frac{1}{N} \left[S - K \widetilde{\lambda} \eps^{p-1} + \O(\eps^{(N-p)/p})\right]^{N/p} + \O(\eps^{(N-p^2)\, q/p\, (p-q)}),
\]
and $(N - p^2)\, q/p\, (p - q) > p - 1$ since $q > (p - 1)\, p^2/(N - p)$.
\end{proof}

\begin{proof}[Proof of Theorem \ref{Theorem 5}]
We apply Lemma \ref{Lemma 2}. Since $q < p < q^\ast$, $W^{1,p}_0(\Omega) \hookrightarrow W^{1,q}_0(\Omega) \hookrightarrow L^p(\Omega)$ by the H\"{o}lder inequality and the Sobolev imbedding, so
\begin{equation} \label{14}
T = \inf_{u \in W^{1,p}_0(\Omega) \setminus \set{0}}\, \frac{\norm[q]{\nabla u}^q}{\norm[p]{u}^q} \ge \inf_{u \in W^{1,q}_0(\Omega) \setminus \set{0}}\, \frac{\norm[q]{\nabla u}^q}{\norm[p]{u}^q} > 0.
\end{equation}
By \eqref{3}, \eqref{6}, and \eqref{14},
\[
\Phi^+(u) \ge \frac{1}{p} \norm{u}^p - \frac{1}{p^\ast}\, S^{- p^\ast/p} \norm{u}^{p^\ast} + \frac{1}{q} \left(1 - \frac{\mu^+}{\mu_1}\right) \norm[q]{\nabla u}^q - \frac{\lambda}{p}\, T^{-p/q} \norm[q]{\nabla u}^p \quad \forall u \in W^{1,p}_0(\Omega),
\]
where $\mu^+ = \max \set{\mu,0}$. Since $\mu^+ < \mu_1$ and $p^\ast > p > q$, it follows from this that $0$ is a strict local minimizer of $\Phi^+$, so $c > 0$. It is clear from Lemma \ref{Lemma 3} \ref{Lemma 3.ii} that $c < S^{N/p}/N$ for $\lambda > 0$ sufficiently large.
\end{proof}

\section{Proof of Theorem \ref{Theorem 6}}

\begin{proof}[Proof of Theorem \ref{Theorem 6}]
Since $q < p$, $W^{1,p}_0(\Omega) \hookrightarrow W^{1,q}_0(\Omega)$ by the H\"{o}lder inequality. Let $S_p$ denote the unit sphere of $W^{1,p}_0(\Omega)$ and let
\[
\pi_p(u) = \frac{u}{\norm[p]{\nabla u}}, \quad u \in W^{1,p}_0(\Omega) \setminus \set{0}, \qquad \pi_q(u) = \frac{u}{\norm[q]{\nabla u}}, \quad u \in W^{1,q}_0(\Omega) \setminus \set{0}
\]
be the radial projections onto $S_p$ and $S_q$, respectively. Since $\mu \ge \mu_1$, $\mu_k \le \mu < \mu_{k+1}$ for some $k \ge 1$. Then
\begin{equation} \label{20}
i(\pi_q^{-1}(\Psi^{\mu_k})) = i(\pi_q^{-1}(S_q \setminus \Psi_{\mu_{k+1}})) = k
\end{equation}
by \eqref{15}. Set $M = \big\{u \in W^{1,q}_0(\Omega) : \norm[q]{u} = 1\big\}$. By Degiovanni and Lancelotti \cite[Theorem 2.3]{MR2514055}, the set $\pi_q^{-1}(\Psi^{\mu_k}) \cup \set{0}$ contains a symmetric cone $C$ such that $C \cap M$ is compact in $C^1(\Omega)$ and
\begin{equation} \label{21}
i(C \setminus \set{0}) = k.
\end{equation}
Since $W^{1,p}_0(\Omega)$ is a dense linear subspace of $W^{1,q}_0(\Omega)$, the inclusion $\pi_q^{-1}(S_q \setminus \Psi_{\mu_{k+1}}) \cap W^{1,p}_0(\Omega) \incl \pi_q^{-1}(S_q \setminus \Psi_{\mu_{k+1}})$ is a homotopy equivalence by Palais \cite[Theorem 17]{MR0189028}, so
\begin{equation} \label{22}
i(\pi_q^{-1}(S_q \setminus \Psi_{\mu_{k+1}}) \cap W^{1,p}_0(\Omega)) = k
\end{equation}
by \eqref{20}. We apply Theorem \ref{Theorem 8} to our functional $\Phi$ defined in \eqref{2} with
\[
A_0 = \pi_p(C \setminus \set{0}) = \pi_p(C \cap M), \qquad B_0 = S_p \setminus (\pi_q^{-1}(S_q \setminus \Psi_{\mu_{k+1}}) \cap W^{1,p}_0(\Omega)),
\]
noting that $A_0$ is compact since $C \cap M$ is compact and $\pi_p$ is continuous. We have
\[
i(A_0) = i(C \setminus \set{0}) = k
\]
by \eqref{21}, and
\[
i(S_p \setminus B_0) = i(\pi_q^{-1}(S_q \setminus \Psi_{\mu_{k+1}}) \cap W^{1,p}_0(\Omega)) = k
\]
by \eqref{22}, so \eqref{16} holds.

For $u \in S_p$ and $t \ge 0$,
\begin{equation} \label{23}
\Phi(tu) \le \frac{t^q}{q} \int_\Omega \big(|\nabla u|^q - \mu\, |u|^q\big)\, dx - \frac{\widetilde{\lambda}\, t^p}{p} \int_\Omega |u|^p\, dx - \frac{t^p}{p} \left(\widetilde{\lambda} \int_\Omega |u|^p\, dx - 1\right),
\end{equation}
where $\widetilde{\lambda} = \lambda/2$. Pick any $v \in S_p \setminus A_0$. Since $A_0$ is compact, so is the set
\[
X_0 = \set{\pi_p((1 - t)\, u + tv) : u \in A_0,\, 0 \le t \le 1}
\]
and hence
\[
\alpha = \inf_{u \in X_0}\, \int_\Omega |u|^p\, dx > 0, \qquad \beta = \sup_{u \in X_0}\, \int_\Omega \big(|\nabla u|^q - \mu\, |u|^q\big)\, dx < \infty.
\]
Let $\lambda \ge 2/\alpha$, so that $\widetilde{\lambda} \alpha \ge 1$. Then for $u \in A_0 \subset X_0$ and $t \ge 0$, \eqref{23} gives
\begin{equation} \label{24}
\Phi(tu) \le - (\mu - \mu_k)\, \frac{t^q}{q} \int_\Omega |u|^q\, dx \le 0
\end{equation}
since $\mu \ge \mu_k$. For $u \in X_0$ and $t \ge 0$, \eqref{23} gives
\begin{equation} \label{25}
\Phi(tu) \le \frac{\beta\, t^q}{q} - \frac{\widetilde{\lambda} \alpha\, t^p}{p} \le \left(\dfrac{1}{q} - \dfrac{1}{p}\right) \dfrac{(\beta^+)^{p/(p-q)}}{(\widetilde{\lambda} \alpha)^{q/(p-q)}},
\end{equation}
where $\beta^+ = \max \set{\beta,0}$. Fix $\lambda$ so large that the last expression is $< S^{N/p}/N$, take positive $R \ge (p\, \beta^+/q\, \widetilde{\lambda} \alpha)^{1/(p-q)}$, and let $A$ and $X$ be as in Theorem \ref{Theorem 8}. Then it follows from \eqref{24} and \eqref{25} that
\[
\sup \Phi(A) \le 0, \qquad \sup \Phi(X) < \frac{S^{N/p}}{N}.
\]
Since $p < q^\ast$, $W^{1,q}_0(\Omega) \hookrightarrow L^p(\Omega)$ by the Sobolev imbedding, so
\begin{equation} \label{26}
T = \inf_{u \in W^{1,p}_0(\Omega) \setminus \set{0}}\, \frac{\norm[q]{\nabla u}^q}{\norm[p]{u}^q} \ge \inf_{u \in W^{1,q}_0(\Omega) \setminus \set{0}}\, \frac{\norm[q]{\nabla u}^q}{\norm[p]{u}^q} > 0.
\end{equation}
By \eqref{3} and \eqref{26},
\[
\Phi(u) \ge \frac{1}{p} \norm{u}^p - \frac{1}{p^\ast}\, S^{- p^\ast/p} \norm{u}^{p^\ast} + \frac{1}{q} \left(1 - \frac{\mu}{\mu_{k+1}}\right) \norm[q]{\nabla u}^q - \frac{\lambda}{p}\, T^{-p/q} \norm[q]{\nabla u}^p \quad \forall u \in \pi_p^{-1}(B_0).
\]
Since $\mu < \mu_{k+1}$ and $p^\ast > p > q$, it follows from this that if $0 < r < R$ is sufficiently small and $B$ is as in Theorem \ref{Theorem 8}, then
\[
\inf \Phi(B) > 0.
\]
Then $0 < c < S^{N/p}/N$ and $\Phi$ has a \PS{c} sequence by Theorem \ref{Theorem 8}, a subsequence of which converges weakly to a nontrivial critical point of $\Phi$ by Proposition \ref{Proposition 1}.
\end{proof}

\section*{Acknowledgement}

\noindent The authors have been partially supported by the Gruppo Nazionale per l'Analisi Matematica, la Probabilit\`{a} e le loro Applicazioni (GNAMPA) of the Istituto Nazionale di Alta Matematica (INdAM).

\def\cdprime{$''$}

\end{document}